\documentclass[12pt]{article}
\usepackage[left=35mm,right=35mm,top=35mm,bottom=35mm]{geometry}
\usepackage{amsmath}
\usepackage{amsthm}
\usepackage{amssymb}
\usepackage{amsfonts}
\usepackage{mathrsfs}

\DeclareMathOperator*{\slim}{s-lim}

\DeclareMathOperator{\supp}{supp}
\DeclareMathOperator{\arcsinh}{arcsinh}

\makeatletter

\@addtoreset{equation}{section}
\makeatother

\newtheorem{ass}{Assumption}[section]
\newtheorem{thm}[ass]{Theorem}
\newtheorem{prop}[ass]{Proposition}
\newtheorem{lem}[ass]{Lemma}
\newtheorem{rem}[ass]{Remark}

\begin{document}
\begin{flushleft}
{\Large \bf Quantum inverse scattering for time-dependent repulsive Hamiltonians of quadratic type}
\end{flushleft}

\begin{flushleft}
{\large Atsuhide ISHIDA}\\
{Katsushika Division, Institute of Arts and Sciences, Tokyo University of Science, 6-3-1 Niijuku, Katsushika-ku, Tokyo 125-8585, Japan\\ 
Email: aishida@rs.tus.ac.jp
}
\end{flushleft}

%abstract%%%%%%%%%%%%%%%%%%%%%%%%%%
\begin{abstract}
We study a multidimensional inverse scattering problem under the time-dependent repulsive Hamiltonians of quadratic type. The time-dependent coefficient on the repulsive term decays as the inverse square of time, which is the threshold between the standard free Schr\"odinger operator and the time-independent repulsive Hamiltonians of quadratic type. Applying the Enss-Weder time-dependent method, we can determine uniquely the short-range potential functions with Coulomb-like singularities from the velocity limit of the scattering operator.
 \end{abstract}


%introduction%%%%%%%%%%%%%%%%%%%%%%%%
%%%%%%%%%%%%%%%%%%%%%%%%%%%%%%%
\section{Introduction\label{introduction}}
Let us consider the quantum system governed by the following time-dependent repulsive Hamiltonian of quadratic type.
\begin{equation}
H_0(t)=p^2/2-k(t)x^2/2\label{free_hamiltonian}
\end{equation}
on $L^2(\mathbb{R}^n)$ with $n\geqslant2$, where $x=(x_1,\ldots,x_n)\in\mathbb{R}^n$ is the position of the particle, $p=-{\rm i}\nabla=-\sqrt{-1}(\partial_{x_1},\ldots,\partial_{x_n})$ with $\partial_{x_j}=\partial/\partial x_j$ for $1\leqslant j\leqslant n$ is the momentum, and the time-dependent coefficient is
\begin{equation}
k(t)=
\begin{cases}
\ \omega^2 & \quad \mbox{if}\quad|t|\leqslant1,\\
\ \sigma/t^2 & \quad \mbox{if}\quad|t|> 1\label{coeff}
\end{cases}
\end{equation}
with $\omega>0$ and $\sigma>0$. The total Hamiltonian $H(t)$ is given by the perturbation of the above $H_0(t)$ by the potential function $V$, that is
\begin{equation}
H(t)=H_0(t)+V
\end{equation}
with $V=V^{\rm sing}+V^{\rm reg}$, where $V^{\rm sing}$ and $V^{\rm reg}$ are defined in Assumption \ref{ass} below. We will use the following notation: $\|\cdot\|$ denotes the $L^2$-norm and operator norm on $L^2(\mathbb{R}^n)$, $(\cdot,\cdot)$ the scalar product of $L^2(\mathbb{R}^n)$, and $F(\cdots)$ the characteristic function of the set $\{\cdots\}$. The bracket $\langle\cdot\rangle$ has the standard definition $\langle x\rangle=\sqrt{1+x^2}$.The inequality $A\lesssim B$ means that there exists a constant $C>0$ such that $A\leqslant CB$.

\begin{ass}\label{ass}
Let $V$ be the multiplication operator as the sum of the real-valued measurable functions $V^{\rm sing}=V^{\rm sing}(x)$ and $V^{\rm reg}=V^{\rm reg}(x)$. The singular part $V^{\rm sing}$ is compactly supported in $\mathbb{R}^n$ and satisfies $V^{\rm sing}\in L^q(\mathbb{R}^n)$ for
\begin{equation}
q
\begin{cases}
\ =2\quad & \mbox{\rm if}\quad n\leqslant3,\\
\ >n/2\quad & \mbox{\rm if}\quad n\geqslant4.\label{sing}
\end{cases}
\end{equation}
The regular part $V^{\rm reg}$ satisfies $V^{\rm reg}\in C^1(\mathbb{R}^n)$ and
\begin{equation}
|\partial_x^\beta V^{\rm reg}(x)|\lesssim
\begin{cases}
\ \langle x\rangle^{-\rho-|\beta|/2}\quad & \mbox{\rm if}\quad\sigma\leqslant2,\\
\ \langle x\rangle^{-\rho-|\beta|}\quad & \mbox{\rm if}\quad\sigma>2\label{reg_decay}
\end{cases}
\end{equation}
for $1/\lambda<\rho<1$ with
\begin{equation}
\lambda=(1+\sqrt{1+4\sigma})/2>1\label{lambda}
\end{equation}
and the multi-indices $\beta\in(\mathbb{N}\cup\{0\})^n$ with $|\beta|\leqslant1$.
\end{ass}

Applying \cite[Theorem 6 and Remark (a)]{Ya}, there exist strong continuous propagators $U_0(t,s)$ and $U(t,s)$ for $H_0(t)$ and $H(t)$, respectively, under Assumption \ref{ass}. That is, $U(t,s)$ satisfies
\begin{itemize}
\item ${\rm i}\partial_tU(t,s)=H(t)U(t,s)$
\item ${\rm i}\partial_sU(t,s)=-U(t,s)H(s)$
\item $U(t,t)=1$
\end{itemize}
for all $t,s\in\mathbb{R}\setminus\{\pm1\}$, where $1$ above is the identity operator of $L^2(\mathbb{R}^n)$, and $U_0(t,s)$ also satisfies the same properties. Incidentally, we find that $x(t)=ct^\lambda$ with $c>0$ satisfies the Newton equation of the classical mechanics
\begin{equation}
({\rm d}^2/{\rm d}t^2)x(t)=k(t)x(t)\label{newton}
\end{equation}
for $t>1$ because $\lambda$ from $\eqref{lambda}$ is the one of the roots of the quadratic equation $\lambda(\lambda-1)=\sigma$. This implies that the critical value between the short- and long-range conditions of the potential function is $-1/\lambda$ and that $V^{\rm reg}$ is short-range. If the negative sign in front of $k(t)$ in \eqref{free_hamiltonian} is made positive, that is
\begin{equation}
\tilde{H}_0(t)=p^2/2+k(t)x^2/2,\label{time_dependent_harmonic}
\end{equation}
this system is called the time-dependent or time-decaying harmonic oscillator. Under this time-dependent quantum system, \cite[Theorem 1]{IsKa1} assumed that the potential function $V\in C(\mathbb{R}^n)$ satisfies
\begin{equation}
|V(x)|\lesssim\langle x\rangle^{-\gamma}
\end{equation}
with $\gamma>0$, and proved that the wave operators exist if $\gamma>1/(1-\mu)$ and do not exist if $\gamma\leqslant1/(1-\mu)$ where
\begin{equation}
\mu=(1-\sqrt{1-4\sigma})/2\label{mu}
\end{equation}
with $\sigma<1/4$. Replacing $1-\mu$ with $\lambda$ in the proof of \cite[Theorem 1]{IsKa1}, we can prove that the wave operators
\begin{equation}
W^{\pm}=\slim_{t\rightarrow\pm\infty}U(t,0)^*U_0(t,0)\label{wave_op}
\end{equation}
exist for $\gamma>1/\lambda$ and do not exist if $\gamma\leqslant1/\lambda$. In \cite[Theorem 1]{IsKa1}, the proof of the existence of the wave operators was given only for the short-range potential function $V\in C(\mathbb{R}^d)$. However, we can prove the existence of \eqref{wave_op} even for $V=V^{\rm sing}+V^{\rm reg}$ under Assumption \ref{ass}, using the Cook-Kuroda method \cite[Theorem XI. 4]{ReSi2} in the same way of \cite{IsKa1}. We therefore can define the scattering operator $S=S(V)$ by
\begin{equation}
S=(W^+)^*W^-.
\end{equation}
The main result in this paper is the following:

\begin{thm}\label{thm1}
Let $V_1$ and $V_2$ satisfy Assumption \ref{ass}. If $S(V_1)=S(V_2)$, then $V_1=V_2$ holds.
\end{thm}

Cases where the coefficient $k(t)$ is time-independent are well-known. For example, if $k(t)$ is identically equal to zero, $p^2/2$ is the standard free Sch\"odinger operator, and there are many known results on scattering theory in this case. If $k(t)$ is an identically negative constant $-\omega^2$, then $p^2/2+\omega^2x^2/2$ is the harmonic oscillator. In the quantum system with this harmonic oscillator, it is well-known that all states are bound states and there are no scattering phenomena. If $k(t)$ is an identically positive constant $\omega^2$, then $p^2/2-\omega^2x^2/2$ is called the repulsive quadratic Hamiltonian or just the repulsive Hamiltonian. In this case, it is known that the classical motion of the particle has exponential behavior in $t$, and the corresponding scattering theory is discussed in \cite{BoCaHaMi,Is1,It1,It2}. For the time-dependent $k(t)=\sigma/t^\nu$, we are particularly interested in the case $\nu=2$ based on the following explanation. If $k(t)=\sigma/t^{\nu}$ with $\nu\not=2$ and $x(t)=ct^\lambda$ with $c>0$ and $\lambda>0$ for $t>1$, we have
\begin{equation}
\lambda(\lambda-1)=\sigma t^{2-\nu}\label{not=2} 
\end{equation}
from the Newton equation \eqref{newton}. Because equation \eqref{not=2} must hold even as $t\rightarrow\infty$, then we have $\lambda(\lambda-1)=0$ and $\lambda=1$ if $\nu>2$. In this case, the classical behavior is $x(t)=O(t)$ as $t\rightarrow\infty$, meaning that the system governed by $H_0(t)$ is asymptotically equivalent to the system governed by the free Schr\"odinger operator $p^2/2$. Whereas if $\nu<2$, we have $\lambda=\infty$ because $\lambda(\lambda-1)=\infty$ as $t\rightarrow\infty$ in \eqref{not=2}. This implies that the classical behavior exhibits exponential growth as $t\rightarrow\infty$ and that system is asymptotically governed by the repulsive quadratic Hamiltonian $p^2/2-\omega^2x^2/2$. These observations imply that the power $\nu=-2$ on $t$ is a threshold. Under the quantum system governed by the time-dependent harmonic oscillator \eqref{time_dependent_harmonic}, the power $\nu=-2$ of $t$ is also a threshold. There is much progress on this model recently in the linear and nonlinear scattering theory \cite{IsKa1,IsKa2,IsKa3,Ka1,Ka2,Ka3,KaMi1,KaMi2,KaMu,KaSa,KaYo}.

The Enss-Weder time-dependent method is originated in Enss and Weder \cite{EnWe} for the standard two- and $N$-body Sch\"odinger operators, and the uniqueness of the potential functions was proved. Since then, this method has been applied to various quantum systems, outer electric fields \cite{AdFuIs,AdKaKaTo,AdMa,AdTs1,AdTs2,Is3,Is6,Ni1,Ni2,VaWe,We}, time-independent repulsive Hamiltonian \cite{Is1,Is5,Ni3}, fractional Laplacian and Dirac equation \cite{Is4,Ju}, time-dependent harmonic oscillator \cite{Is7}, the nonlinear Schr\"odinger equation \cite{Wa1}, and Hartree-Fock equation \cite{Wa2,Wa3}.

This time-dependent repulsive quadratic Hamiltonian has never been investigated before this paper, although the time-dependent harmonic oscillator has been discussed by many authors. This time-dependent model does not relate to any concrete physics phenomena, however is an interesting model from a mathematical perspective. As we stated above, if $H_0(t)$ is time-independent, $H_0(t)=H_0\equiv p^2/2-\omega^2x^2/2$, and the behavior of the particle thus has the exponential growth as $t\rightarrow\infty$. The work in \cite{BoCaHaMi} focused on this behavior, proving that the wave operators are asymptotically complete if $V$ has the space decay $V(x)=O((\log|x|)^{-\gamma})$ with $\gamma>1$. It was later proved that the wave operators do not exist if $\gamma\leqslant1$ by \cite{Is2}. In the contrast, the choice of time-dependent coefficient $k(t)$ can lead to very different results. We can intuitively understand that $x^2$ and $t^{-2}$ balance each other, and this balance decreases the classical behavior to $O(t^\lambda)$ as $t\rightarrow\infty$.

The time evolution of $e^{-{\rm i}tH_0}$ can be written by as the composite mappings of the unitary operators as follows, called the Mehler formula, introduced by \cite{Ho}. It is well known that the Mehler formula for the harmonic oscillator involves trigonometric functions in the representation of the propagator. In contrast, for the repulsive version, the corresponding Mehler formula involves hyperbolic functions:
\begin{equation}
e^{-{\rm i}tH_0}=\mathscr{M}(\tanh\omega t/\omega)\mathscr{D}(\sinh\omega t/\omega)\mathscr{F}\mathscr{M}(\tanh\omega t/\omega)\label{mehler1}%Mehler formula
\end{equation}
where $\mathscr{M}$ and $\mathscr{D}$ denote multiplication and dilation, with
\begin{gather}
\mathscr{M}(t)\phi(x)=e^{{\rm i}x^2/(2t)}\phi(x),\\
\mathscr{D}(t)\phi(x)=(2{\rm i}t)^{-n/2}\phi(x/t),
\end{gather}
and $\mathscr{F}$ denotes the Fourier transform over $L^2(\mathbb{R}^n)$. For simplicity, we use the notation $\sinh(\omega t)=\sinh\omega t$, and similarly for the other hyperbolic functions. We also have
\begin{gather}
\mathscr{D}(\sinh\omega t/\omega)={\rm i}^{n/2}\mathscr{D}(\cosh\omega t)\mathscr{D}(\tanh\omega t/\omega),\\
\mathscr{M}(\tanh\omega t/\omega)\mathscr{D}(\cosh\omega t)\mathscr{M}(-\tanh\omega t/\omega)=\mathscr{M}(\coth\omega t/\omega)\mathscr{D}(\cosh\omega t)
\end{gather}
and
\begin{equation}
e^{-{\rm i}tH_0}={\rm i}^{n/2}\mathscr{M}(\coth\omega t/\omega)\mathscr{D}(\cosh\omega t)e^{-{\rm i}\tanh\omega tp^2/(2\omega)}\label{mehler2}
\end{equation}
from the well-known formula for the standard free propagator
\begin{equation}
e^{-{\rm i}tp^2/2}=\mathscr{M}(t)\mathscr{D}(t)\mathscr{F}\mathscr{M}(t).
\end{equation}
We will frequently use \eqref{mehler2} instead of \eqref{mehler1}, which was first introduced in \cite{Is1} (see also \cite{Is5, Is7}). The time evolution $U_0(t,s)$ can be also decomposed to the unitary mappings:

\begin{lem}\label{lem0}
Let $U_{0,\lambda}(t)$ be defined by
\begin{equation}
U_{0,\lambda}(t)=e^{-{\rm i}(\lambda-1) x^2/(2t)}e^{{\rm i}(\lambda-1)\log tA}e^{-{\rm i}t^{2\lambda-1}p^2/(2(2\lambda-1))}\label{propagator_lambda}
\end{equation}
if $t>1$ and
\begin{equation}
U_{0,\lambda}(t)=e^{-{\rm i}(\lambda-1) x^2/(2t)}e^{{\rm i}(\lambda-1)\log(-t)A}e^{-{\rm i}(-t)^{2\lambda-1}p^2/(2(2\lambda-1))}
\end{equation}
if $t<-1$ where $A=(p\cdot x+x\cdot p)/2$. Then
\begin{equation}
U_0(s,t)=U_{0,\lambda}(t)U_{0,\lambda}(s)^*.\label{unitary_decomposition}
\end{equation}
holds for $t,s>1$ or $t,s<-1$.
\end{lem}

\begin{proof}[Proof of Lemma \ref{lem0}]
It was proved by \cite[Proposition 1]{IsKa1} that the propagator $\tilde{U}_0(t,s)$ generated by $\tilde{H}_0(t)$ can be written as
\begin{equation}
\tilde{U}_0(t,s)=U_{0,1-\mu}(t)U_{0,1-\mu}(s)^*
\end{equation}
for $t,s>1$ or $t,s<-1$. Replacing $1-\mu$ with $\lambda$ in the proof of \cite[Proposition 1]{IsKa1}, we have ${\rm i}\partial_tU_{0,\lambda}(t)=H_0(t)U_{0,\lambda}(t)$ for $|t|>1$ and ${\rm i}\partial_sU_{0,\lambda}(s)^*=-U_{0,\lambda}(s)^*H_0(s)$ for $|s|>1$. The uniqueness of the propagators implies that \eqref{unitary_decomposition} holds for $t,s>1$ or $t,s<-1$.
\end{proof}

Moreover, we define
\begin{equation}
U_{0,\lambda}(t)=e^{-{\rm i}tH_0}
\end{equation}
for $|t|\leqslant1$. Because the wave operators \eqref{wave_op} exist, the strong limits
\begin{equation}
W_\lambda^\pm=\slim_{t\rightarrow\pm\infty}U(t,0)^*U_{0,\lambda}(t)
\end{equation}
also exist and we define $S_\lambda=S_\lambda(V)$ by
\begin{equation}
S_\lambda=(W_\lambda^+)^*W_\lambda^-.
\end{equation}
Noting that $W^\pm=W_\lambda^\pm U_{0,\lambda}(s_\pm)^*U_0(s_\pm,0)$ for $s_+>1$ and $s_-<-1$, we have the following relation between $S$ and $S_\lambda$:
\begin{equation}
S=U_0(s_+,0)^*U_{0,\lambda}(s_+)S_\lambda U_{0,\lambda}(s_-)^*U_0(s_-,0),
\end{equation}
and we find that $S(V_1)=S(V_2)$ is equivalent to $S_\lambda(V_1)=S_\lambda(V_2)$. The idea to use $W_\lambda^\pm$ and $S_\lambda$ instead of $W^\pm$ and $S$ was first introduced by \cite{Is7} for the inverse problem under the time-dependent harmonic oscillator. Investigating the time evolution $U_0(t,0)$ directly is thought to be difficult. However, the decomposition using $U_{0,\lambda}(t)$ reduces it to the time evolution $e^{-{\rm i}tH_0}$ if $|t|\leqslant1$ and $e^{\mp{\rm i}|t|^{2\lambda-1}p^2/(2(2\lambda-1))}$ if $|t|>1$.

Our proof considers two cases: $\sigma\leqslant2$ and $\sigma>2$. The exponent $\rho>1/2$ holds in \eqref{reg_decay} if $\sigma\leqslant2$, and this decay condition is same as for Stark short-range. Alternatively, $\rho\leqslant1/2$ if $\sigma>2$, which is Stark long-range (see \cite{JeOz, Oz}). If we directly apply the proof for the case $\sigma\leqslant2$ to the case $\sigma>2$, we do not obtain sufficient decay of $|v|$ as required in \eqref{thm2_1}. Therefore, we define scalar-valued modified wave operators in the manner of the Graf- and Zorbas-types \cite{Gr, Zo} for $\sigma>2$ to restore this weak decay to \eqref{thm2_9}. Graf-type (or Zorbas-type) modified wave operators were first introduced in the inverse problem by \cite{AdMa} even for Stark short-range potentials. Since then, this idea has been adopted by many studies on inverse problems \cite{AdFuIs,AdKaKaTo,AdTs1,AdTs2,Is1,Is3,Is5,VaWe}. Moreover, we need the stronger decay condition for the first derivative in \eqref{reg_decay} for the case $\sigma>2$ because integrable decay is required regarding $t$ in the reconstructing limit \eqref{reconstructing_limit} (see also \eqref{thm2_5}, \eqref{thm2_7}, \eqref{thm2_10} and \eqref{thm2_11}).

%reconstructing limit%%%%%%%%%%%%%%%%%%%%%%%%
%%%%%%%%%%%%%%%%%%%%%%%%%%%%%%%
\section{Reconstructing Limit}
The following result is essential for the proof of Theorem \ref{thm1}. We therefore devote ourselves to proving this limit, and give the complete proof of Theorem \ref{thm1} at the end of this section.

By proving the following reconstructing limit, the proof of Theorem \ref{thm1} is demonstrated as in \cite[Theorem 1.2]{We} (see also \cite{EnWe}) by the Plancherel formula of the Radon transform \cite[Theorem 2.17 in Chapter 1]{He}. The commutator of the scattering operator with $p_j$ was first introduced to the reconstructing limit by \cite{We} for the Stark effect. This idea was also adopted in \cite{AdFuIs,AdKaKaTo,AdTs1,AdTs2, Is1, Is3, Is5, Ni1, Ni2, Ni3, VaWe}.

\begin{thm}\label{thm2}
Let $\Phi_0\in\mathscr{S}(\mathbb{R}^n)$ be such that $\mathscr{F}\Phi_0\in C_0^\infty(\mathbb{R}^n)$. For $v\in\mathbb{R}^n$, its normalization is $\hat{v}=v/|v|$. Let $\Phi_v=e^{{\rm i}v\cdot x}\Phi_0$ and $\Psi_v$ have the same properties. Then
\begin{gather}
\lim_{|v|\rightarrow\infty}|v|({\rm i}[S_\lambda,p_j]\Phi_v,\Psi_v)\nonumber\\
=\int_{-\infty}^\infty\left\{\right.(V^{\rm sing}(x+\hat{v}t)p_j\Phi_0,\Psi_0)-(V^{\rm sing}(x+\hat{v}t)\Phi_0,p_j\Psi_0)\nonumber\\
+({\rm i}(\partial_{x_j}V^{\rm reg})(x+\hat{v}t)\Phi_0,\Psi_0)\left.\right\}{\rm d}t\label{reconstructing_limit}
\end{gather}
holds for $1\leqslant j\leqslant n$ where $p_j$ is the $j$th component of $p$.
\end{thm}

The following proposition, proved by \cite[Proposition 2.10]{En}, is useful for the proofs of some estimates used in turn to prove Theorem \ref{thm2}:

\begin{prop}\label{enss}
Let $M$ and $M'$ be measurable subsets of $\mathbb{R}^n$ and let $f\in C_0^\infty(\mathbb{R}^n)$ have $\supp f\subset\{\xi\in\mathbb{R}^n\bigm||\xi|\leqslant\eta\}$ for some $\eta>0$. Then
\begin{equation}
\|F(x\in M')e^{-{\rm i}tp^2/2}f(p)F(x\in M)\|\lesssim_{N,f}(1+|t|+r)^{-N}
\end{equation}
for $t\in\mathbb{R}$ and $N\in\mathbb{N}$, where $r={\rm dist}(M',M)-\eta|t|\mathbb{R}^n\geqslant0$ and $\lesssim_{N,f}$ means that the constant depends on $N$ and $f$.
\end{prop}

We first prove the following propagation estimate for $V^{\rm sing}$:

\begin{lem}\label{lem1}
Let $\Phi_v$ be as in Theorem \ref{thm2}. Then
\begin{equation}
\int_{-\infty}^\infty\|V^{\rm sing}(x)U_{0,\lambda}(t)\Phi_v\|{\rm d}t=O(|v|^{-1})
\end{equation}
holds as $|v|\rightarrow\infty$.
\end{lem}

\begin{proof}[Proof of Lemma \ref{lem1}]
The strategy of this proof is originated in \cite[Lemmas 2.3 and 3.2]{Is7}. Let $f\in C_0^\infty(\mathbb{R}^n)$ satisfy $\Phi_0=f(p)\Phi_0$ and $\supp f\subset\{\xi\in\mathbb{R}^n\bigm||\xi|\leqslant\eta\}$ for some $\eta>0$. We separate the integral into two parts $|t|\leqslant1$ and $|t|>1$, so that
\begin{equation}
\int_{-\infty}^{\infty}=\int_{|t|\leqslant1}+\int_{|t|>1}
\end{equation}
and first consider $|t|\leqslant1$. By \eqref{mehler2} and
\begin{equation}
e^{-{\rm i}v\cdot x}e^{-{\rm i}\tanh\omega tp^2/(2\omega)}e^{-{\rm i}v\cdot x}=e^{-{\rm i}\tanh\omega t|v|^2/(2\omega)}e^{-{\rm i}\tanh\omega tp\cdot v/\omega}e^{-{\rm i}\tanh\omega tp^2/(2\omega)},\label{lem1_1}
\end{equation}
we have
\begin{gather}
\|V^{\rm sing}(x)e^{-{\rm i}tH_0}\Phi_v\|=\|V^{\rm sing}(\cosh\omega tx)e^{-{\rm i}\tanh\omega tp^2/(2\omega)}\Phi_v\|\nonumber\\
=\|V^{\rm sing}(\cosh\omega tx+\sinh\omega tv/\omega)\langle p/\cosh\omega t\rangle^{-2}e^{-{\rm i}\tanh\omega tp^2/(2\omega)}\langle p/\cosh\omega t\rangle^2\Phi_0\|\nonumber\\
\leqslant I_1+I_2+I_3,
\end{gather}
where
\begin{align}
I_1=&\|V^{\rm sing}(\cosh\omega tx)\langle p/\cosh\omega t\rangle^{-2}\|\|F(|x|\geqslant|\tanh\omega t||v|/(2\omega))\nonumber\\
&\times e^{-{\rm i}\tanh\omega tp^2/(2\omega)}f(p)F(|x|\leqslant|\tanh\omega t||v|/(4\omega))\|\|\langle p/\cosh\omega t\rangle^2\Phi_0\|,\nonumber\\
I_2=&\|V^{\rm sing}(\cosh\omega tx)\langle p/\cosh\omega t\rangle^{-2}\|\|F(|x|\geqslant|\tanh\omega t||v|/(2\omega))e^{-{\rm i}\tanh\omega tp^2/(2\omega)}\nonumber\\
&\times f(p)F(|x|>|\tanh\omega t||v|/(4\omega))\langle x\rangle^{-2}\|\|\langle x\rangle^2\langle p/\cosh\omega t\rangle^2\Phi_0\|,\nonumber\\
I_3=&\|V^{\rm sing}(\cosh\omega tx+\sinh\omega tv/\omega)\langle p/\cosh\omega t\rangle^{-2}F(|x|<|\tanh\omega t||v|/(2\omega))\|\nonumber\\
&\times \|\langle p/\cosh\omega t\rangle^2\Phi_0\|\label{lem1_2}
\end{align}
as in the proof of \cite[Lemma 3.2]{Is7}. We note that
\begin{equation}
\|V^{\rm sing}(\cosh\omega tx)\langle p/\cosh\omega t\rangle^{-2}\|=\|V^{\rm sing}(x)\langle p\rangle^{-2}\|
\end{equation}
and
\begin{equation}
\|\langle p/\cosh\omega t\rangle^2\Phi_0\|\leqslant\|\langle x\rangle^2\langle p/\cosh\omega t\rangle^2\Phi_0\|\lesssim1
\end{equation}
because $|t|\leqslant1$. By Proposition \ref{enss} for $I_1$, we have
\begin{gather}
\int_{|t|\leqslant1}(I_1+I_2){\rm d}t\lesssim\int_0^1\langle\tanh\omega tv\rangle^{-2}{\rm d}t\nonumber\\
\lesssim\int_0^1\langle tv\rangle^{-2}{\rm d}t=|v|^{-1}\int_0^{|v|}\langle \tau\rangle^{-2}{\rm d}\tau=O(|v|^{-1}),\label{lem1_3}
\end{gather}
where we used $|\tanh\omega t|\geqslant\omega|t|/2$ and $\tau=|v|t$. For $|x|<|\tanh\omega t||v|/(2\omega)$,
\begin{equation}
|\cosh\omega tx+\sinh\omega tv/\omega|>|\sinh\omega t||v|/(2\omega)\geqslant|t||v|/2\label{lem1_4}
\end{equation}
holds and we have
\begin{gather}
\int_{|t|\leqslant1}I_3{\rm d}t\lesssim\int_0^1\|V^{\rm sing}(x)\langle p\rangle^{-2}F(|x|\geqslant t|v|/2)\|{\rm d}t\nonumber\\
=|v|^{-1}\int_0^1+|v|^{-1}\int_1^{|v|}\|V^{\rm sing}(x)\langle p\rangle^{-2}F(|x|\geqslant\tau/2)\|{\rm d}\tau.\label{lem1_5}
\end{gather}
To estimate this integral from $1$ to $|v|$, we take $\chi\in C^\infty(\mathbb{R}^n)$ such that $\chi(x)=1$ if $|x|\geqslant1$ and $\chi(x)=0$ if $|x|\leqslant1/2$ and we have
\begin{gather}
\|V^{\rm sing}(x)\langle p\rangle^{-2}F(|x|\geqslant\tau/2)\|\leqslant\|V^{\rm sing}(x)\langle p\rangle^{-2}\chi(2x/\tau)\|\nonumber\\
\lesssim\|V^{\rm sing}(x)\chi(2x/\tau)\langle p\rangle^{-2}\|+\tau^{-1}\|V^{\rm sing}(x)(\nabla\chi)(2x/\tau)\langle p\rangle^{-2}\|+\tau^{-2}\label{lem1_6}
\end{gather}
by computation of the commutator $[\langle p\rangle^{-2},\chi(2x/\tau)]$. Because $V^{\rm sing}$ is compactly supported, the integral intervals of the first and second terms in \eqref{lem1_6} are finite. We thus have
\begin{equation}
\int_{|t|\leqslant1}I_3{\rm d}t=O(|v|^{-1}).\label{lem1_7}
\end{equation}
We next consider the integral on $|t|>1$, in particular for $t>1$ because the case $t<-1$ can be handled similarly. By \eqref{propagator_lambda}, $e^{-{\rm i}(\lambda-1)\log tA}x_je^{{\rm i}(\lambda-1)\log tA}=t^{1-\lambda}x_j$, and
\begin{gather}
e^{-{\rm i}v\cdot x}e^{-{\rm i}t^{2\lambda-1}p^2/(2(2\lambda-1))}e^{{\rm i}v\cdot x}\nonumber\\
=e^{-{\rm i}t^{2\lambda-1}|v|^2/(2(2\lambda-1))}e^{-{\rm i}t^{2\lambda-1}p\cdot v/(2\lambda-1)}e^{-{\rm i}t^{2\lambda-1}p^2/(2(2\lambda-1))},
\end{gather}
we have
\begin{gather}
\|V^{\rm sing}(x)U_{0,\lambda}(t)\Phi_v\|=\|V^{\rm sing}(t^{1-\lambda}x)e^{-{\rm i}t^{2\lambda-1}p^2/(2(2\lambda-1))}\Phi_v\|\nonumber\\
=\|V^{\rm sing}(t^{1-\lambda}x+t^\lambda v/(2\lambda-1))\langle t^{\lambda-1}p\rangle^{-2}e^{-{\rm i}t^{2\lambda-1}p^2/(2(2\lambda-1))}\langle t^{\lambda-1}p\rangle^2\Phi_0\|\nonumber\\
\leqslant I_4+I_5+I_6,
\end{gather}
where
\begin{align}
I_4=&\|V^{\rm sing}(t^{1-\lambda}x)\langle t^{\lambda-1}p\rangle^{-2}\|\|F(|x|\geqslant t^{2\lambda-1}|v|/(2(2\lambda-1)))\nonumber\\
&\times e^{-{\rm i}t^{2\lambda-1}p^2/(2(2\lambda-1))}f(p)F(|x|\leqslant t^{2\lambda-1}|v|/(4(2\lambda-1)))\|\|\langle t^{\lambda-1}p\rangle^2\Phi_0\|,\nonumber\\
I_5=&\|V^{\rm sing}(t^{1-\lambda}x)\langle t^{\lambda-1}p\rangle^{-2}\|\|F(|x|\geqslant t^{2\lambda-1}|v|/(2(2\lambda-1)))e^{-{\rm i}t^{2\lambda-1}p^2/(2(2\lambda-1))}\nonumber\\
&\times f(p)F(|x|>t^{2\lambda-1}|v|/(4(2\lambda-1)))\langle x\rangle^{-2}\|\|\langle x\rangle^2\langle t^{\lambda-1}p\rangle\Phi_0\|,\nonumber\\
I_6=&\|V^{\rm sing}(t^{1-\lambda}x+t^\lambda v/(2\lambda-1))\langle t^{\lambda-1}p\rangle^{-2}F(|x|<t^{2\lambda-1}|v|/(2(2\lambda-1)))\|\nonumber\\
&\times\|\langle t^{\lambda-1}p\rangle^2\Phi_0\|
\end{align}
as in \eqref{lem1_2}. We note that $\|V^{\rm sing}(t^{1-\lambda}x)\langle t^{\lambda-1}p\rangle^{-2}\|=\|V^{\rm sing}(x)\langle p\rangle^{-2}\|$ and
\begin{equation}
\|\langle t^{\lambda-1}p\rangle^2\Phi_0\|\leqslant\|\langle x\rangle^2\langle t^{\lambda-1}p\rangle^2\Phi_0\|\lesssim t^{2\lambda-2}.
\end{equation}
By Proposition \ref{enss} for $I_4$, we have
\begin{gather}
\int_1^\infty(I_4+I_5){\rm d}t\lesssim\int_1^\infty t^{2\lambda-2}\langle t^{2\lambda-1}v\rangle^{-2}{\rm d}t\nonumber\\
=(|v|^{-1}/(2\lambda-1))\int_{|v|}^\infty\langle \tau\rangle^{-2}{\rm d}\tau=O(|v|^{-2})\label{lem1_8}
\end{gather}
where we used $\tau=t^{2\lambda-1}|v|$. For $|x|<t^{2\lambda-1}|v|/(2(2\lambda-1))$,
\begin{equation}
|t^{1-\lambda}x+t^\lambda v/(2\lambda-1)|>t^\lambda|v|/(2(2\lambda-1))\label{lem1_9}
\end{equation}
holds and we have
\begin{gather}
\int_1^\infty I_6{\rm d}t\lesssim\int_1^\infty t^{2\lambda-2}\|V^{\rm sing}(x)\langle p\rangle^{-2}F(|x|>t^\lambda|v|/(2(2\lambda-1)))\|{\rm d}t\nonumber\\
=(|v|^{-2+1/\lambda}/\lambda)\int_{|v|}^\infty\tau^{1-1/\lambda}\|V^{\rm sing}(x)\langle p\rangle^{-2}F(|x|>\tau/(2(2\lambda-1)))\|{\rm d}\tau\label{lem1_10}
\end{gather}
where we used $\tau=t^\lambda|v|$. Because $V^{\rm sing}$ is compactly supported, we have
\begin{equation}
\|V^{\rm sing}(x)\langle p\rangle^{-2}F(|x|>\tau/(2(2\lambda-1)))\|\lesssim\tau^{-2}
\end{equation}
by \eqref{lem1_6} for $\tau\geqslant|v|\gg1$. We thus have
\begin{equation}
\int_1^\infty I_6{\rm d}t\lesssim|v|^{-2+1/\lambda}\int_{|v|}^\infty\tau^{-1-1/\lambda}{\rm d}\tau=O(|v|^{-2}).\label{lem1_11}
\end{equation}
Combining \eqref{lem1_3}, \eqref{lem1_7}, \eqref{lem1_8}, and \eqref{lem1_11} completes the proof.
\end{proof}

We now present the proofs for the two cases mentioned in the Introduction.

\subsection{Case of $\sigma\leqslant2$}

\begin{lem}\label{lem2}
Let $\Phi_v$ be as in Theorem \ref{thm2}. Then
\begin{equation}
\int_{-\infty}^\infty\|V^{\rm reg}(x)U_{0,\lambda}(t)\Phi_v\|{\rm d}t=O(|v|^{-\rho})\label{lem2_1}
\end{equation}
holds as $|v|\rightarrow\infty$.
\end{lem}

\begin{proof}[Proof of Lemma \ref{lem2}]
This proof can be demonstrated similarly to the proof of Lemma \ref{lem1} (see also \cite[Lemma 2.3]{Is7}). For the integral on $|t|\leqslant1$, we do not need to insert the resolvent $\langle p/\cosh\omega t\rangle^{-2}$ because $V^{\rm reg}$ is bounded. We therefore have
\begin{equation}
\|V^{\rm reg}(x)e^{-{\rm i}tH_0}\Phi_v\|\leqslant I_1+I_2+I_3,
\end{equation}
where
\begin{align}
I_1=&\|V^{\rm reg}(x)\|\|F(|x|\geqslant|\tanh\omega t||v|/(2\omega))\nonumber\\
&\times e^{-{\rm i}\tanh\omega tp^2/(2\omega)}f(p)F(|x|\leqslant|\tanh\omega t||v|/(4\omega))\|\|\Phi_0\|,\nonumber\\
I_2=&\|V^{\rm reg}(x)\|\|F(|x|\geqslant|\tanh\omega t||v|/(2\omega))e^{-{\rm i}\tanh\omega tp^2/(2\omega)}\nonumber\\
&\times f(p)F(|x|>|\tanh\omega t||v|/(4\omega))\langle x\rangle^{-2}\|\|\langle x\rangle^2\Phi_0\|,\nonumber\\
I_3=&\|V^{\rm reg}(\cosh\omega tx+\sinh\omega tv/\omega)F(|x|<|\tanh\omega t||v|/(2\omega))\|\|\Phi_0\|.
\end{align}
We have
\begin{equation}
\int_{|t|\leqslant1}(I_1+I_2){\rm d}t=O(|v|^{-1})
\end{equation}
as in the proof of Lemma \ref{lem1}. By the assumption \eqref{reg_decay} and same computation using \eqref{lem1_5}, $I_3$ is estimated as
\begin{equation}
\int_{|t|\leqslant1}I_3{\rm d}t\lesssim|v|^{-1}\int_0^{|v|}\langle \tau\rangle^{-\rho}{\rm d}\tau=O(|v|^{-\rho})\label{lem2_2}
\end{equation}
where we used $\rho<1$. As for the integral on $|t|>1$, we do not need to insert the resolvent $\langle t^{\lambda-1}p\rangle^{-2}$. We therefore have
\begin{equation}
\|V^{\rm reg}(x)U_{0,\lambda}(t)\Phi_v\|\leqslant I_4+I_5+I_6
\end{equation}
where
\begin{align}
I_4=&\|V^{\rm reg}(x)\|\|F(|x|\geqslant t^{2\lambda-1}|v|/(2(2\lambda-1)))\nonumber\\
&\times e^{-{\rm i}t^{2\lambda-1}p^2/(2(2\lambda-1))}f(p)F(|x|\leqslant t^{2\lambda-1}|v|/(4(2\lambda-1)))\|\|\Phi_0\|,\nonumber\\
I_5=&\|V^{\rm reg}(x)\|\|F(|x|\geqslant t^{2\lambda-1}|v|/(2(2\lambda-1)))e^{-{\rm i}t^{2\lambda-1}p^2/(2(2\lambda-1))}\nonumber\\
&\times f(p)F(|x|>t^{2\lambda-1}|v|/(4(2\lambda-1)))\langle x\rangle^{-2}\|\|\langle x\rangle^2\Phi_0\|,\nonumber\\
I_6=&\|V^{\rm reg}(t^{1-\lambda}x+t^\lambda v/(2\lambda-1))F(|x|<t^{2\lambda-1}|v|/(2(2\lambda-1)))\|\|\Phi_0\|.
\end{align}
By the same computations using \eqref{lem1_8} and \eqref{lem1_10}, $I_4$, $I_5$, and $I_6$ are estimated such that
\begin{gather}
\int_1^\infty(I_4+I_5){\rm d}t\lesssim\int_1^\infty\langle t^{2\lambda-1}v\rangle^{-2}{\rm d}t\nonumber\\
=(|v|^{-1/(2\lambda-1)})\int_{|v|}^\infty\tau^{(2-2\lambda)/(2\lambda-1)}\langle\tau\rangle^{-2}{\rm d}\tau=O(|v|^{-2}),\label{lem2_3}\\
\int_1^\infty I_6{\rm d}t\lesssim\int_1^\infty\langle t^\lambda v\rangle^{-\rho}{\rm d}t=(|v|^{-1/\lambda}/\lambda)\int_{|v|}^\infty\tau^{1/\lambda-1}\langle \tau\rangle^{-\rho}{\rm d}\tau=O(|v|^{-\rho})\label{lem2_4}
\end{gather}
because $(2-2\lambda)/(2\lambda-1)-2<-1$ and $1/\lambda-1-\rho<-1$.
\end{proof}

\begin{rem}
{\rm In the assumption regarding the size of $\rho$, the upper bound $\rho<1$ is not essential. We impose this upper bound because our interest is in the case where $\rho$ is close to $1/\lambda$. For the general $\rho>1/\lambda$ without this upper bound, \eqref{lem2_1} is replaced with
\begin{equation}
\int_{-\infty}^\infty\|V^{\rm reg}(x)U_{0,\lambda}(t)\Phi_v\|{\rm d}t=
\begin{cases}
\ O(|v|^{-\rho})\quad & \mbox{if}\quad \rho<1,\\
\ O(|v|^{-1}\log|v|)\quad & \mbox{if}\quad \rho=1,\\
\ O(|v|^{-1})\quad & \mbox{if}\quad \rho>1
\end{cases}
\end{equation}
by \eqref{lem2_2}.}
\end{rem}

\begin{lem}\label{lem3}
Let $\Phi_v$ be as in Theorem \ref{thm2}. Then
\begin{equation}
\sup_{t\in\mathbb{R}}\|(U(t,0)W_\lambda^--U_{0,\lambda}(t))\Phi_v\|=O(|v|^{-\rho})
\end{equation}
holds as $|v|\rightarrow\infty$.
\end{lem}

\begin{proof}[Proof of Lemma \ref{lem3}]
This proof is originated in \cite[Corollay 2.3]{EnWe} (see also \cite{Is4, Is7, Ni3, We}). We have
\begin{gather}
W_\lambda^--U(t,0)^*U_{0,\lambda}(t)=-\int_{-\infty}^t({\rm d}/{\rm d}\tau)U(\tau,0)^*U_{0,\lambda}(\tau){\rm d}\tau\nonumber\\
=-{\rm i}\int_{-\infty}^tU(\tau,0)^*V(x)U_{0,\lambda}(\tau){\rm d}\tau
\end{gather}
and
\begin{gather}
\|(W_\lambda^--U(t,0)^*U_{0,\lambda}(t))\Phi_v\|\nonumber\\
\leqslant\int_{-\infty}^\infty\|V^{\rm sing}(x)U_{0,\lambda}(t)\Phi_v\|{\rm d}\tau+\int_{-\infty}^\infty\|V^{\rm reg}(x)U_{0,\lambda}(t)\Phi_v\|{\rm d}\tau.
\end{gather}
Lemmas \ref{lem1} and \ref{lem2} complete this proof.
\end{proof}

\begin{proof}[Proof of Theorem \ref{thm2} for $\sigma\leqslant2$]
Noting that $[S_\lambda,p_j]=[S_\lambda-1,p_j-v_j]$, $(p_j-v_j)\Phi_v=(p_j\Phi_0)_v$, and
\begin{equation}
{\rm i}(S_\lambda-1)={\rm i}(W_\lambda^+-W_\lambda^-)^*W_\lambda^-=\int_{-\infty}^\infty U_{0,\lambda}(t)^*V(x)U(t,0)W_\lambda^-{\rm d}t,
\end{equation}
we have
\begin{equation}
|v|({\rm i}[S_\lambda,p_j]\Phi_v,\Psi_v)=I(v)+R(v),
\end{equation}
where
\begin{align}
I(v)&=|v|\int_{-\infty}^\infty\left\{\right.(V(x)U_{0,\lambda}(t)(p_j\Phi_0)_v,U_{0,\lambda}(t)\Psi_v)\nonumber\\
&\quad-(V(x)U_{0,\lambda}(t)\Phi_v,U_{0,\lambda}(t)(p_j\Psi_0)_v)\left.\right\}{\rm d}t,\\
R(v)&=|v|\int_{-\infty}^\infty\left\{\right.((U(t,0)W_\lambda^--U_{0,\lambda}(t))(p_j\Phi_0)_v,V(x)U_{0,\lambda}(t)\Psi_v)\nonumber\\
&\quad-((U(t,0)W_\lambda^--U_{0,\lambda}(t))\Phi_v,V(x)U_{0,\lambda}(t)(p_j\Psi_0)_v)\left.\right\}{\rm d}t.
\end{align}
Lemmas \ref{lem1}, \ref{lem2}, and \ref{lem3} imply that
\begin{equation}
R(v)=O(|v|^{1-2\rho})\label{thm2_1}
\end{equation}
and that this converges to zero as $|v|\rightarrow\infty$ because $\rho>1/2$ for $\sigma\leqslant2$. We now prove that $I(v)$ converges to the right-hand side of \eqref{reconstructing_limit}. We separate the integral into two parts, $|t|\leqslant1$ and $|t|>1$, so that
\begin{equation}
\int_{-\infty}^\infty=\int_{|t|\leqslant1}+\int_{|t|>1}
\end{equation}
and first consider $|t|\leqslant1$. Using \eqref{mehler2} and \eqref{lem1_1}, we have
\begin{equation}
e^{-{\rm i}v\cdot x}e^{{\rm i}tH_0}V(x)e^{-{\rm i}tH_0}e^{{\rm i}v\cdot x}=e^{{\rm i}tH_0}V(x+\sinh\omega tv/\omega)e^{-{\rm i}tH_0}.
\end{equation}
We thus have
\begin{gather}
(V(x)e^{-{\rm i}tH_0}(p_j\Phi_0)_v,e^{-{\rm i}tH_0}\Psi_v)-(V(x)e^{-{\rm i}tH_0}\Phi_v,e^{-{\rm i}tH_0}(p_j\Psi_0)_v)\nonumber\\
=(V^{\rm sing}(x+\sinh\omega tv/\omega)e^{-{\rm i}tH_0}p_j\Phi_0,e^{-{\rm i}tH_0}\Psi_0)\nonumber\\
\quad-(V^{\rm sing}(x+\sinh\omega tv/\omega)e^{-{\rm i}tH_0}\Phi_0,e^{-{\rm i}tH_0}p_j\Psi_0)\nonumber\\
\quad+\cosh\omega t({\rm i}(\partial_{x_j}V^{\rm reg})(x+\sinh\omega tv/\omega)e^{-{\rm i}tH_0}\Phi_0,e^{-{\rm i}tH_0}\Psi_0),\label{thm2_2}
\end{gather}
where we used
\begin{equation}
e^{{\rm i}tH_0}p_je^{-{\rm i}tH_0}=\omega\sinh\omega tx_j+\cosh\omega tp_j\label{thm2_3}
\end{equation}
obtained by \cite[Lemma 6]{Ni3}. The integral of the first term on the right-hand side of \eqref{thm2_2} is
\begin{gather}
|v|\int_{|t|\leqslant1}(V^{\rm sing}(x+\sinh\omega tv/\omega)e^{-{\rm i}tH_0}p_j\Phi_0,e^{-{\rm i}tH_0}\Psi_0){\rm d}t=\int_{|\tau|\leqslant\sinh\omega|v|/\omega}\langle\omega\tau/|v|\rangle^{-1}\nonumber\\
\quad\times(V^{\rm sing}(x+\hat{v}\tau)e^{-{\rm i}\arcsinh(\omega\tau/|v|)H_0/\omega}p_j\Phi_0,e^{-{\rm i}\arcsinh(\omega\tau/|v|)H_0/\omega}\Psi_0){\rm d}\tau.
\end{gather}
Because $e^{-{\rm i}tH_0}$ is strongly continuous at $t=0$, we have
\begin{gather}
\langle\omega\tau/|v|\rangle^{-1}(V^{\rm sing}(x+\hat{v}\tau)e^{-{\rm i}\arcsinh(\omega\tau/|v|)H_0/\omega}p_j\Phi_0,e^{-{\rm i}\arcsinh(\omega\tau/|v|)H_0/\omega}\Psi_0)\nonumber\\
\rightarrow(V^{\rm sing}(x+\hat{v}\tau)p_j\Phi_0,\Psi_0)
\end{gather}
as $|v|\rightarrow\infty$ for any $\tau\in\mathbb{R}$ noting that
\begin{equation}
\|\langle p\rangle^2e^{-{\rm i}\arcsinh(\omega\tau/|v|)H_0/\omega}\Phi_0\|\lesssim1
\end{equation}
by \eqref{thm2_3}. We now have
\begin{gather}
|v|\int_{|t|\leqslant1}|(V^{\rm sing}(x)e^{-{\rm i}tH_0}(p_j\Phi_0)_v,e^{-{\rm i}tH_0}\Psi_v)|{\rm d}t\nonumber\\
=\int_{|\tau|\leqslant|v|}|(V^{\rm sing}(x)e^{-{\rm i}(\tau/|v|)H_0}(p_j\Phi_0)_v,e^{-{\rm i}(\tau/|v|)H_0}\Psi_v)|{\rm d}\tau
\end{gather}
by $\tau=|v|t$. As we showed in the proof of Lemma \ref{lem1}, the integrand above is estimated as
\begin{gather}
|(V^{\rm sing}(x)e^{-{\rm i}(\tau/|v|)H_0}(p_j\Phi_0)_v,e^{-{\rm i}(\tau/|v|)H_0}\Psi_v)|\leqslant\|V^{\rm sing}(x)e^{-{\rm i}(\tau/|v|)H_0}(p_j\Phi_0)_v\|\|\Psi_0\|\nonumber\\
\lesssim\langle\tau\rangle^{-2}+\|V^{\rm sing}(x)\langle p\rangle^{-2}F(|x|\geqslant|\tau|/2)\|.\label{thm2_4}
\end{gather}
The right-hand side of \eqref{thm2_4} is integrable for $\tau$ independently of $v$. By the Lebesgue dominated convergence theorem, we have
\begin{equation}
|v|\int_{|t|\leqslant1}(V^{\rm sing}(x)e^{-{\rm i}tH_0}(p_j\Phi_0)_v,e^{-{\rm i}tH_0}\Psi_v){\rm d}t\rightarrow\int_{-\infty}^\infty(V^{\rm sing}(x+\hat{v}t)p_j\Phi_0,\Psi_0){\rm d}t
\end{equation}
as $|v|\rightarrow\infty$. Similarly, we have
\begin{equation}
|v|\int_{|t|\leqslant1}(V^{\rm sing}(x)e^{-{\rm i}tH_0}\Phi_v,e^{-{\rm i}tH_0}(p_j\Psi_0)_v){\rm d}t\rightarrow\int_{-\infty}^\infty(V^{\rm sing}(x+\hat{v}t)\Phi_0,p_j\Psi_0){\rm d}t
\end{equation}
as $|v|\rightarrow\infty$. The integral of the third term on the right-hand side of \eqref{thm2_2} is
\begin{gather}
|v|\int_{|v|\leqslant1}\cosh\omega t({\rm i}(\partial_{x_j}V^{\rm reg})(x+\sinh\omega tv/\omega)e^{-{\rm i}tH_0}\Phi_0,e^{-{\rm i}tH_0}\Psi_0){\rm d}t\nonumber\\
=\int_{|\tau|\leqslant\sinh\omega|v|/\omega}({\rm i}(\partial_{x_j}V^{\rm reg})(x+\hat{v}\tau)e^{-{\rm i}\arcsinh(\omega\tau/|v|)H_0/\omega}\Phi_0,e^{-{\rm i}\arcsinh(\omega\tau/|v|)H_0/\omega}\Psi_0){\rm d}\tau.
\end{gather}
Clearly,
\begin{gather}
({\rm i}(\partial_{x_j}V^{\rm reg})(x+\hat{v}\tau)e^{-{\rm i}\arcsinh(\omega\tau/|v|)H_0/\omega}\Phi_0,e^{-{\rm i}\arcsinh(\omega\tau/|v|)H_0/\omega}\Psi_0)\nonumber\\
\rightarrow({\rm i}(\partial_{x_j}V^{\rm reg})(x+\hat{v}\tau)\Phi_0,\Psi_0)
\end{gather}
as $|v|\rightarrow\infty$ holds for any $\tau\in\mathbb{R}$. We have
\begin{gather}
|v|\int_{|t|\leqslant1}\cosh\omega t|({\rm i}(\partial_{x_j}V^{\rm reg})(x)e^{-{\rm i}tH_0}\Phi_v,e^{-{\rm i}tH_0}\Psi_v)|{\rm d}t\nonumber\\
=\int_{|\tau|\leqslant|v|}\cosh(\omega\tau/|v|)|({\rm i}(\partial_{x_j}V^{\rm reg})(x)e^{-{\rm i}(\tau/|v|)H_0}\Phi_v,e^{-{\rm i}(\tau/|v|)H_0}\Psi_v)|{\rm d}\tau
\end{gather}
and also have
\begin{gather}
\cosh(\omega\tau/|v|)|({\rm i}(\partial_{x_j}V^{\rm reg})(x)e^{-{\rm i}(\tau/|v|)H_0}\Phi_v,e^{-{\rm i}(\tau/|v|)H_0}\Psi_v)|\nonumber\\
\leqslant\cosh\omega\|(\partial_{x_j}V^{\rm reg})(x)e^{-{\rm i}(\tau/|v|)H_0}\Phi_v\|\|\Psi_0\|\lesssim\langle\tau\rangle^{-\rho-1/2}\label{thm2_5}
\end{gather}
by \eqref{reg_decay} and same computations using \eqref{lem1_3} and \eqref{lem2_2}. Because $-\rho-1/2<-1$ for $\sigma\leqslant2$, we have
\begin{gather}
|v|\int_{|t|\leqslant1}\cosh\omega t({\rm i}(\partial_{x_j}V^{\rm reg})(x)e^{-{\rm i}tH_0}\Phi_v,e^{-{\rm i}tH_0}\Psi_v){\rm d}t\nonumber\\
\rightarrow\int_{-\infty}^\infty({\rm i}(\partial_{x_j}V^{\rm reg})(x+\hat{v}t)\Phi_0,\Psi_0){\rm d}t
\end{gather}
as $|v|\rightarrow\infty$ by the Lebesgue dominated convergence theorem. We next consider the integral on $|t|>1$. By \eqref{lem1_8} and \eqref{lem1_11}, we have
\begin{gather}
|v|\int_{|t|>1}\left\{\right.(V^{\rm sing}(x)U_{0,\lambda}(t)(p_j\Phi_0)_v,U_{0,\lambda}(t)\Psi_v)\nonumber\\
-(V^{\rm sing}(x)U_{0,\lambda}(t)\Phi_v,U_{0,\lambda}(t)(p_j\Psi_0)_v)\left.\right\}{\rm d}t=O(|v|^{-1}).\label{thm2_6}
\end{gather}
Using $U_{0,\lambda}(t)p_jU_{0,\lambda}(t)^*=p_j/|t|^{\lambda-1}+(1-\lambda)x_j/(t|t|^{\lambda-1})$, we have
\begin{gather}
|v|\int_{|t|>1}\left\{\right.(V^{\rm reg}(x)U_{0,\lambda}(t)(p_j\Phi_0)_v,U_{0,\lambda}(t)\Psi_v)\nonumber\\
-(V^{\rm reg}(x)U_{0,\lambda}(t)\Phi_v,U_{0,\lambda}(t)(p_j\Psi_0)_v)\left.\right\}{\rm d}t\nonumber\\
=|v|\int_{|t|>1}|t|^{1-\lambda}({\rm i}(\partial_{x_j}V^{\rm reg})(x)U_{0,\lambda}(t)\Phi_v,U_{0,\lambda}(t)\Psi_v){\rm d}t.
\end{gather}
By same computations using with \eqref{lem2_3} and \eqref{lem2_4}, we have
\begin{gather}
|v|\int_{|t|>1}|t|^{1-\lambda}|({\rm i}(\partial_{x_j}V^{\rm reg})(x)U_{0,\lambda}(t)\Phi_v,U_{0,\lambda}(t)\Psi_v)|{\rm d}t\nonumber\\
\lesssim|v|\int_1^\infty\left\{t^{1-\lambda}\langle t^{2\lambda-1}v\rangle^{-2}+t^{1-\lambda}\langle t^\lambda v\rangle^{-\rho-1/2}\right\}{\rm d}t\nonumber\\
=(|v|^{(3\lambda-3)/(2\lambda-1)}/(2\lambda-1))\int_{|v|}^\infty\tau^{(3-3\lambda)/(2\lambda-1)}\langle\tau\rangle^{-2}{\rm d}\tau\nonumber\\
+(|v|^{2-2/\lambda}/\lambda)\int_{|v|}^\infty\tau^{2/\lambda-2}\langle\tau\rangle^{-\rho-1/2}{\rm d}\tau=O(|v|^{-1})+O(|v|^{-\rho+1/2}).\label{thm2_7}
\end{gather}
Here, we used $(3-3\lambda)/(2\lambda-1)-2<-1$ and $2/\lambda-2-\rho-1/2<-1$. Estimates \eqref{thm2_6} and \eqref{thm2_7} imply that the integral of $I(v)$ on $|t|>1$ converges to zero as $|v|\rightarrow\infty$ because $\rho>1/2$ for $\sigma\leqslant2$. This completes the proof.
\end{proof}

\begin{rem}
{\rm We can prove Theorem \ref{thm2} in the case $\sigma\leqslant2$ for the mixed-type potential $V_{\rm sing}$ which satisfies $V_{\rm sing}\in L^q(\mathbb{R}^n)$ with \eqref{sing} and $\langle x\rangle^\rho V_{\rm sing}(x)\langle p\rangle^{-2}$ is the bounded operator on $L^2(\mathbb{R}^n)$ for $\rho>1/\lambda$. Indeed, Lemma \ref{lem1} holds for $V_{\rm sing}$ by slight modifications of the estimates $I_3$ and $I_6$, and the proof of Theorem \ref{thm2} for $\sigma\leqslant2$ is demonstrated in the same way.
}
\end{rem}

\subsection{Case of $\sigma>2$}
As stated in the introduction, $R(v)$ in \eqref{thm2_1} does not converge to zero if we directly apply the proof for the case $\sigma\leqslant2$ to the case for $\sigma>2$. To overcome this difficulty, we define the classical trajectory of the particle
\begin{equation}
\alpha(t)=
\begin{cases}
\ \sinh\omega t/\omega\quad & \mbox{if}\quad|t|\leqslant1,\\
\ t^\lambda/(2\lambda-1)\quad & \mbox{if}\quad t>1,\\
\ -(-t)^\lambda/(2\lambda-1)\quad & \mbox{if}\quad t<-1
\end{cases}\label{alpha}
\end{equation}
and improve the decay for $|v|$ in Lemma \ref{lem2} to the following:

\begin{lem}\label{lem4}
Let $\Phi_v$ be as in Theorem \ref{thm2}. Then
\begin{equation}
\int_{-\infty}^\infty\|\{V^{\rm reg}(x)-V^{\rm reg}(\alpha(t)v)\}U_{0,\lambda}(t)\Phi_v\|{\rm d}t=O(|v|^{-1})
\end{equation}
holds as $|v|\rightarrow\infty$.
\end{lem}

\begin{proof}[Proof of Lemma \ref{lem4}]
As in the proofs of Lemmas \ref{lem1} and \ref{lem2}, we separate the integral to $|t|\leqslant1$ and $|t|>1$. For the integral on $|t|\leqslant1$, we have
\begin{equation}
\|\{V^{\rm reg}(x)-V^{\rm reg}(\sinh\omega tv/\omega)\}e^{-{\rm i}tH_0}\Phi_v\|\leqslant I_1+I_2+I_3,
\end{equation}
where
\begin{align}
I_1=&2\|V^{\rm reg}(x)\|\|F(|x|\geqslant|\tanh\omega t||v|/(2\omega))\nonumber\\
&\times e^{-{\rm i}\tanh\omega tp^2/(2\omega)}f(p)F(|x|\leqslant|\tanh\omega t||v|/(4\omega))\|\|\Phi_0\|,\nonumber\\
I_2=&2\|V^{\rm reg}(x)\|\|F(|x|\geqslant|\tanh\omega t||v|/(2\omega))e^{-{\rm i}\tanh\omega tp^2/(2\omega)}\nonumber\\
&\times f(p)F(|x|>|\tanh\omega t||v|/(4\omega))\langle x\rangle^{-2}\|\|\langle x\rangle^2\Phi_0\|,\nonumber\\
I_3=&\|\{V^{\rm reg}(\cosh\omega tx+\sinh\omega tv/\omega)-V^{\rm reg}(\sinh\omega tv/\omega)\}\nonumber\\
&\times F(|x|<|\tanh\omega t||v|/(2\omega))e^{-{\rm i}\tanh\omega tp^2/(2\omega)}\Phi_0\|,
\end{align}
noting that we keep $e^{-{\rm i}\tanh\omega tp^2/(2\omega)}$ in $I_3$. We have
\begin{equation}
\int_{|t|\leqslant1}(I_1+I_2){\rm d}t=O(|v|^{-1})\label{lem4_1}
\end{equation}
as in the proof of Lemma \ref{lem1}. As for $I_3$, we have
\begin{gather}
I_3\leqslant\int_0^1\|(\nabla_x V^{\rm reg})(\cosh\omega t\theta x+\sinh\omega tv/\omega)\cdot\cosh\omega t x\nonumber\\
\times F(|x|<|\tanh\omega t||v|/(2\omega))e^{-{\rm i}\tanh\omega tp^2/(2\omega)}\Phi_0\|{\rm d}\theta.
\end{gather}
We note the fact
\begin{equation}
e^{{\rm i}\tanh\omega tp^2/(2\omega)}x_je^{-{\rm i}\tanh\omega tp^2/(2\omega)}=x_j+\tanh\omega tp/\omega
\end{equation}
and that
\begin{equation}
\|x_je^{-{\rm i}\tanh\omega tp^2/(2\omega)}\Phi_0\|\lesssim1\label{lem4_2}
\end{equation}
for $1\leqslant j\leqslant n$. Using \eqref{reg_decay}, \eqref{lem1_4}, and \eqref{lem4_2}, we have
\begin{gather}
\int_{|t|\leqslant1}I_3{\rm d}t\lesssim\int_0^1\langle \sinh\omega tv/\omega\rangle^{-\rho-1}\cosh\omega t{\rm d}t\nonumber\\
=|v|^{-1}\int_0^{\sinh\omega|v|/\omega}\langle \tau\rangle^{-\rho-1}{\rm d}\tau=O(|v|^{-1})\label{lem4_3}
\end{gather}
by $\tau=\sinh\omega tv/\omega$. The idea of using \eqref{lem4_2} and \eqref{lem4_5} below originate from \cite{Is3}. For the integral on $|t|>1$, in particular for $t>1$, we have
\begin{equation}
\|\{V^{\rm reg}(x)-V^{\rm reg}(t^\lambda v/(2\lambda-1))\}U_{0,\lambda}(t)\Phi_v\|\leqslant I_4+I_5+I_6,
\end{equation}
where
\begin{align}
I_4=&2\|V^{\rm reg}(x)\|\|F(|x|\geqslant t^{2\lambda-1}|v|/(2(2\lambda-1)))\nonumber\\
&\times e^{-{\rm i}t^{2\lambda-1}p^2/(2(2\lambda-1))}f(p)F(|x|\leqslant t^{2\lambda-1}|v|/(4(2\lambda-1)))\|\|\Phi_0\|,\nonumber\\
I_5=&2\|V^{\rm reg}(x)\|\|F(|x|\geqslant t^{2\lambda-1}|v|/(2(2\lambda-1)))e^{-{\rm i}t^{2\lambda-1}p^2/(2(2\lambda-1))}\nonumber\\
&\times f(p)F(|x|>t^{2\lambda-1}|v|/(4(2\lambda-1)))\langle x\rangle^{-2}\|\|\langle x\rangle^2\Phi_0\|,\nonumber\\
I_6=&\|\{V^{\rm reg}(t^{1-\lambda}x+t^\lambda v/(2\lambda-1))-V^{\rm reg}(t^\lambda v/(2\lambda-1))\}\nonumber\\
&\times F(|x|<t^{2\lambda-1}|v|/(2(2\lambda-1)))e^{-{\rm i}t^{2\lambda-1}p^2/(2(2\lambda-1))}\Phi_0\|,
\end{align}
noting that we keep $e^{-{\rm i}t^{2\lambda-1}p^2/(2(2\lambda-1))}$ in $I_6$. We have
\begin{equation}
\int_{|t|\leqslant1}(I_4+I_5){\rm d}t=O(|v|^{-1})\label{lem4_4}
\end{equation}
as in the proof of Lemma \ref{lem1}. Using \eqref{reg_decay} again, \eqref{lem1_9}, and
\begin{equation}
\|x_je^{-{\rm i}t^{2\lambda-1}p^2/(2(2\lambda-1))}\Phi_0\|\lesssim t^{2\lambda-1}\label{lem4_5}
\end{equation}
for $1\leqslant j\leqslant n$, which follows from
\begin{equation}
e^{{\rm i}t^{2\lambda-1}p^2/(2(2\lambda-1))}x_je^{-{\rm i}t^{2\lambda-1}p^2/(2(2\lambda-1))}=x_j+t^{2\lambda-1}p/(\lambda-1),
\end{equation}
we have
\begin{gather}
I_6\leqslant\int_0^1\|(\nabla_xV^{\rm reg})(t^{1-\lambda}\theta x+t^\lambda v/(2\lambda-1))\cdot t^{1-\lambda}x\nonumber\\
\times F(|x|<t^{2\lambda-1}|v|/(2(2\lambda-1)))e^{-{\rm i}t^{2\lambda-1}p^2/(2(2\lambda-1))}\Phi_0\|{\rm d}\theta
\end{gather}
and
\begin{gather}
\int_{|t|>1}I_6{\rm d}t\lesssim\int_1^\infty\langle t^{\lambda}v\rangle^{-\rho-1}t^\lambda{\rm d}t=(|v|^{-1-1/\lambda}/\lambda)\int_{|v|}^\infty\tau^{1/\lambda}\langle\tau\rangle^{-\rho-1}{\rm d}\tau=O(|v|^{-\rho-1})\label{lem4_6}
\end{gather}
by $\tau=t^\lambda|v|$, noting that $1/\lambda-\rho-1<-1$. Combining \eqref{lem4_1}, \eqref{lem4_3}, \eqref{lem4_4}, and \eqref{lem4_6} completes the proof.
\end{proof}

We here introduce the Graf-type modified wave operators
\begin{equation}
\Omega_{\lambda,v}^{\pm}=\slim_{t\rightarrow\pm\infty}U(t,0)^*U_{0,\lambda}(t)e^{-{\rm i}\int_0^tV^{\rm reg}(\alpha(\tau)v){\rm d}\tau}.
\end{equation}
More precisely, although the Graf-type was originally formulated in the setting of outer electric fields, we will still refer to it as the Graf-type in this paper. By the assumption \eqref{reg_decay} and $|\sinh\omega t|\geqslant\omega|t|$, we have
\begin{gather}
\int_{-\infty}^\infty|V^{\rm reg}(\alpha(t)v)|{\rm d}t=\int_{|t|\leqslant1}+\int_{|t|>1}|V^{\rm reg}(\alpha(t)v)|{\rm d}t\nonumber\\
\lesssim\int_0^1\langle vt\rangle^{-\rho}{\rm d}t+\int_1^\infty\langle vt^\lambda\rangle^{-\rho}{\rm d}t<\infty,\label{graf}
\end{gather}
noting $\lambda\rho>1$. This implies that
\begin{equation}
\Omega_{\lambda,v}^{\pm}=W_\lambda^\pm e^{-{\rm i}\int_0^{\pm\infty}V^{\rm reg}(\alpha(\tau)v){\rm d}\tau}
\end{equation}
exist. As an alternative candidates for the modified wave operators, \cite{Ni1,Ni2} adopted the Dollard-type for the inverse scattering under the Stark effect and time-periodic electric fields. The Dollard-type is defined using the Fourier multiplier. Later, \cite{AdMa} first adopted the Graf-type under the Stark effect. Since then, the Graf-type modification has been the standard method for inverse scattering. In contrast to the Dollard-type, we emphasize that the Graf-type is scalar-valued.

\begin{lem}\label{lem5}
Let $\Phi_v$ be as in Theorem \ref{thm2}. Then
\begin{equation}
\sup_{t\in\mathbb{R}}\|(U(t,0)\Omega_{\lambda,v}^--U_{0,\lambda}(t)e^{-{\rm i}\int_0^tV^{\rm reg}(\alpha(\tau)v){\rm d}\tau})\Phi_v\|=O(|v|^{-1})
\end{equation}
holds as $|v|\rightarrow\infty$.
\end{lem}

\begin{proof}[Proof of Lemma \ref{lem5}]
This proof is originated in \cite[Lemma 2.3]{AdMa} (see also \cite{AdFuIs,AdKaKaTo,AdTs1,AdTs2,Is1,Is3,Is5,VaWe}). For simplicity, we denote
\begin{equation}
U_{{\rm G},\lambda,v}(t)=U_{0,\lambda}(t)e^{-{\rm i}\int_0^tV^{\rm reg}(\alpha(\tau)v){\rm d}\tau}
\end{equation}
 and
 \begin{equation}
 V_v(t,x)=V^{\rm sing}(x)+V^{\rm reg}-V^{\rm reg}(\alpha(t)v)
 \end{equation}
 and have
\begin{gather}
\Omega_{\lambda,v}^--U(t,0)^*U_{{\rm G},\lambda,v}(t)=-\int_{-\infty}^t({\rm d}/{\rm d}\tau)U(\tau,0)^*U_{{\rm G},\lambda,v}(\tau){\rm d}\tau\nonumber\\
=-{\rm i}\int_{-\infty}^tU(\tau,0)^*V_v(\tau,x)U_{{\rm G},\lambda,v}(\tau){\rm d}\tau,
\end{gather}
as in the proof of Lemma \ref{lem3}. We thus have
\begin{gather}
\|(\Omega_{\lambda,v}^--U(t,0)^*U_{{\rm G},\lambda,v}(t))\Phi_v\|\nonumber\\
\leqslant\int_{-\infty}^\infty\|V^{\rm sing}(x)U_{0,\lambda}(t)\Phi_v\|{\rm d}t+\int_{-\infty}^\infty\|\{V^{\rm reg}(x)-V^{\rm reg}(\alpha(t)v)\}U_{0,\lambda}(t)\Phi_v\|{\rm d}t,
\end{gather}
noting that the Graf-type modifier is a scalar. Lemmas \ref{lem1} and \ref{lem4} complete the proof.
\end{proof}

\begin{proof}[Proof of Theorem \ref{thm2} for $\sigma>2$]
Let $I_v=e^{-{\rm i}\int_{-\infty}^\infty V^{\rm reg}(\alpha(t)v){\rm d}t}$. By \eqref{graf}, we have $I_v\rightarrow1$ as $|v|\rightarrow\infty$. Because $\Omega_{\lambda,v}^{\pm}=W_\lambda^\pm e^{-{\rm i}\int_0^{\pm\infty}V^{\rm reg}(\alpha(\tau)v){\rm d}\tau}$, we can write $S_\lambda$ as $S_\lambda=I_v(\Omega_{\lambda,v}^+)^*\Omega_{\lambda,v}^-$ and
\begin{equation}
{\rm i}(S_\lambda-I_v)={\rm i}I_v(\Omega_{\lambda,v}^+-\Omega_{\lambda,v}^-)^*\Omega_{\lambda,v}^-={\rm i}I_v\int_{-\infty}^\infty U_{{\rm G},\lambda,v}(t)^*V_v(t,x)U(t,0)\Omega_{\lambda,v}^-{\rm d}t.\label{thm2_8}
\end{equation}
Representation \eqref{thm2_8} originated in \cite{AdMa}. Noting that $[S_\lambda,p_j]=[S_\lambda-I_v,p_j-v_j]$, we have
\begin{equation}
|v|({\rm i}[S_\lambda,p_j]\Phi_v,\Psi_v)=I_v(I_{\rm G}(v)+R_{\rm G}(v)),
\end{equation}
where
\begin{align}
I_{\rm G}(v)&=|v|\int_{-\infty}^\infty\left\{\right.(V_v(t,x)U_{{\rm G},\lambda,v}(t)(p_j\Phi_0)_v,U_{{\rm G},\lambda,v}(t)\Psi_v)\nonumber\\
&\quad-(V_v(t,x)U_{{\rm G},\lambda,v}(t)\Phi_v,U_{{\rm G},\lambda,v}(t)(p_j\Psi_0)_v)\left.\right\}{\rm d}t,\\
R_{\rm G}(v)&=|v|\int_{-\infty}^\infty\left\{\right.((U(t,0)\Omega_{\lambda,v}^--U_{{\rm G},\lambda,v}(t))(p_j\Phi_0)_v,V(x)U_{0,\lambda}(t)\Psi_v)\nonumber\\
&\quad-((U(t,0)\Omega_{\lambda,v}^--U_{{\rm G},\lambda,v}(t))\Phi_v,V_v(t,x)U_{{\rm G},\lambda,v}(t)(p_j\Psi_0)_v)\left.\right\}{\rm d}t.
\end{align}
By Lemmas \ref{lem1}, \ref{lem4}, and \ref{lem5}, we have
\begin{equation}
R_{\rm G}(v)=O(|v|^{-1})\label{thm2_9}
\end{equation}
as $|v|\rightarrow\infty$. The proof that $I_{\rm G}(v)$ converges to the right-hand side of \eqref{reconstructing_limit} can be demonstrated as in the proof of Theorem \ref{thm2} for $\sigma\leqslant2$. Indeed, $I_{\rm G}(v)=I(v)$ holds because the Graf-type modifier is scalar-valued. By \eqref{reg_decay} with $\sigma>2$, the estimate corresponding to \eqref{thm2_5} is
\begin{equation}
\cosh(\omega\tau/|v|)|({\rm i}(\partial_{x_j}V^{\rm reg})(x)e^{-{\rm i}(\tau/|v|)H_0}\Phi_v,e^{-{\rm i}(\tau/|v|)H_0}\Psi_v)|\lesssim\langle\tau\rangle^{-\rho-1}\label{thm2_10}
\end{equation}  
and \eqref{thm2_7} is
\begin{gather}
|v|\int_{|t|>1}|t|^{1-\lambda}|({\rm i}(\partial_{x_j}V^{\rm reg})(x)U_{0,\lambda}(t)\Phi_v,U_{0,\lambda}(t)\Psi_v)|{\rm d}t\nonumber\\
\lesssim|v|\int_1^\infty\left\{t^{1-\lambda}\langle t^{2\lambda-1}v\rangle^{-2}+t^{1-\lambda}\langle t^\lambda v\rangle^{-\rho-1}\right\}{\rm d}t\nonumber\\
=O(|v|^{-1})+(|v|^{2-2/\lambda}/\lambda)\int_{|v|}^\infty\tau^{2/\lambda-2}\langle\tau\rangle^{-\rho-1}{\rm d}\tau=O(|v|^{-1})+O(|v|^{-\rho}).\label{thm2_11}
\end{gather}
\end{proof}

\begin{proof}[Proof of Theorem \ref{thm1}]
This proof is demonstrated as in \cite[Theorem 1.2]{We} (see also \cite{EnWe}). Let $V=V_1-V_2$, $V^{\rm sing}=V_1^{\rm sing}-V_2^{\rm sing}$ and $V^{\rm reg}=V_1^{\rm reg}-V_2^{\rm reg}$. For $1\leqslant j,k\leqslant n$ with $j\not=k$, we define $y=y_1e_j+y_2e_k\in\mathbb{R}^2$ where $e_j$ and $e_k$ are orthonormal vectors in $\mathbb{R}^n$ and
\begin{equation}
f(y)=(V^{\rm sing}p_j\Phi(y),\Psi(y))-(V^{\rm sing}\Phi(y),p_j\Psi(y))
+({\rm i}(\partial_{x_j}V^{\rm reg})\Phi(y),\Psi(y))
\end{equation}
where $\Phi(y)=e^{-{\rm i}p\cdot y}\Phi$ for $\Phi\in\mathscr{S}(\mathbb{R}^n)$ with $\mathscr{F}\Phi\in C_0^\infty(\mathbb{R}^n)$ and $\Psi(y)$ has same properties. By the assumption \eqref{reg_decay} and the fact that $V^{\rm sing}$ satisfies the Enss condition
\begin{equation}
\int_0^\infty\|V^{\rm sing}(x)\langle p\rangle^{-2}F(|x|\geqslant R)\|{\rm d}R<\infty\label{enss_condition}
\end{equation}
(see \eqref{lem1_6}), we find that $f\in L^2(\mathbb{R}^2)$ and that $f(y)$ is continuous and bounded. Therefore, by Theorem \ref{thm2}, the Radon transform of $f$ satisfies
\begin{equation}
\mathscr{R}f(\hat{v},y)=\int_{-\infty}^\infty f(y+\hat{v}t){\rm d}t\equiv0
\end{equation}
for any $\hat{v}$ in the plane spanned by $e_j$ and $e_k$, and we have $f(y)\equiv0$ by the Plancherel formula of the Radon transform \cite[Theorem 2.17 in Chapter 1]{He}. The identity $\partial_{y_1}(V\Phi(y),\Psi(y))=-{\rm i}f(y)\equiv0$ implies that $(V\Phi(y),\Psi(y))$ does not depend on $y_1$. By \eqref{reg_decay} and \eqref{enss_condition} again, we have
\begin{equation}
(V\Phi,\Psi)=\lim_{|y_1|\rightarrow\infty}(V\Phi(y_1,0),\Psi(y_1,0))=0.
\end{equation}
We conclude that $V\langle p\rangle^{-2}=0$ as an operator and that $V(x)=0$ a.e. $x\in\mathbb{R}^n$.
\end{proof}

%acknowledgments%%%%%%%%%%%%%%%%%%%%%
\noindent\textbf{Acknowledgments.} 
This work was supported by JSPS KAKENHI Grant Numbers JP20K03625 and JP21K03279. Part of this work was done while the author was visiting the Alfr\'ed R\'enyi Institute of Mathematics. The author thanks Haruya Mizutani of Osaka University and referees for their valuable comments.
%references%%%%%%%%%%%%%%%%%%%%%%%%%
%%%%%%%%%%%%%%%%%%%%%%%%%%%%%%%


\begin{thebibliography}{99}
\bibitem{AdFuIs}T. Adachi, Y. Fujiwara, A. Ishida, On multidimensional inverse scattering in time-dependent electric fields, \textit{Inverse Problems} \textbf{29} (2013), 085012, 24 pp.
\bibitem{AdKaKaTo}T. Adachi, T. Kamada, M. Kazuno, K, Toratani, On multidimensional inverse scattering in an external electric field asymptotically zero in time, \textit{Inverse Problems} \textbf{27} (2011), 065006, 17 pp.
\bibitem{AdMa}T. Adachi, K. Maehara, On multidimensional inverse scattering for Stark Hamiltonians, \textit{J. Math. Phys.} \textbf{48} (2007), 042101, 12 pp.
\bibitem{AdTs1}T. Adachi, Y. Tujii, On multidimensional inverse scattering under the time-dependent Stark effect, \textit{Kyshu J. Math.} \textbf{78} (2024), no. 2, 337--372.
\bibitem{AdTs2}T. Adachi, Y. Tujii, On multidimensional inverse scattering under the Stark effect, \textit{Hiroshima Math. J.} \textbf{54} (2024), no. 3, 319--358.
\bibitem{BoCaHaMi}J. F. Bony, R. Carles, D. H\"afner, L. Michel, Scattering theory for the Schr\"odinger equation with repulsive potential, \textit{J. Math. Pures Appl.} (9) \textbf{84} (2005), no. 5, 509--579.
%\bibitem{CyFrKiSi}H. L. Cycon, R. G. Froese, W. Kirsch, B. Simon, Schr\"odinger operators with application to quantum mechanics and global geometry. \textit{Springer-Verlag, Berlin,} 1987.
%\bibitem{PhViLa}B. Philip, T. Vidar, W. Lars B., Besov spaces and applications to difference methods for initial value problems. Lecture Notes in Mathematics, Vol. 434. \textit{Springer-Verlag, Berlin-New York,} 1975.
\bibitem{En}V. Enss, Propagation properties of quantum scattering states, \textit{J. Funct. Anal.} \textbf{52} (1983), no. 2, 219--251.
\bibitem{EnWe}V. Enss, R. Weder, The geometrical approach to multidimensional inverse scattering, \textit{J. Math. Phys.} \textbf{36} (1995), no. 8, 3902--3921.
\bibitem{Gr}G M, Graf, A remark on long-range Stark scattering, \textit{Helv. Phys. Acta} \textbf{64} (1991), no. 7, 1167--1174.
\bibitem{He}S. Helgason, Groups and Geometric Analysis, Integral geometry, invariant differential operators, and spherical functions, Pure and Applied Mathematics, 113. \textit{Academic Press, Inc., Orland, FL}, 1984.
\bibitem{Ho}L. H\"ormander, Symplectic classification of quadratic forms, and general Mehler formulas, \textit{Math. Z.} \textbf{219} (1995), no. 3, 413--449.
\bibitem{Is1}A. Ishida, On inverse scattering problem for the Schr\"odinger equation with repulsive potentials, \textit{J. Math. Phys.} \textbf{55} (2014), no. 8, 082101, 12 pp.
\bibitem{Is2}A. Ishida, The borderline of the short-range condition for the repulsive Hamiltonian, \textit{J. Math. Anal. Appl.} \textbf{438} (2016), no. 1, 267--273.
\bibitem{Is3}A. Ishida, Inverse scattering in the Stark effect, \textit{Inverse Problems} \textbf{35} (2019) no. 10. 105010, 20 pp.
\bibitem{Is4}A. Ishida, Propagation property and application to inverse scattering for fractional powers of negative Laplacian, \textit{East Asian J. Appl. Math.} \textbf{10} (2020), no. 1, 106--122.
\bibitem{Is5}A. Ishida, Inverse scattering for repulsive potential and strong singular interactions, \textit{J. Math. Phys.} \textbf{65} (2024), no. 8, Paper No. 082101, 10 pp.
\bibitem{Is6}A. Ishida, Corrigendum: Inverse scattering in the Stark effect (2019 Inverse Problems 35 105010), \textit{Inverse Problems} \textbf{40} (2024) no. 12, Paper No. 129501, 3 pp.
\bibitem{Is7}A. Ishida, Quantum inverse scattering for time-decaying harmonic oscillators, \textit{Inverse Probl. Imaging} \textbf{19} (2025), no. 2, 282--296.
\bibitem{IsKa1}A. Ishida, M. Kawamoto, Existence and nonexistence of wave operators for time-decaying harmonic oscillators, \textit{Rep. Math. Phys.} \textbf{85} (2020), no. 3, 335--350.
\bibitem{IsKa2}A. Ishida, M. Kawamoto, Critical scattering in a time-dependent harmonic oscillator, \textit{J. Math. Anal. Appl.} \textbf{492} (2020), no. 2, 124475, 9 pp.
\bibitem{IsKa3}A. Ishida, M. Kawamoto, On Schr\"odinger equation with square and inverse-square potentials, \textit{Partial Differ. Equ. Appl.}, in press.
\bibitem{It1}K. Itakura, Stationary scattering theory for repulsive Hamiltonians, \textit{J. Math. Phys.} \textbf{62} (2021), no. 6, Paper No. 061504, 24 pp.
\bibitem{It2}K. Itakura, Limiting absorption principle and radiation condition for repulsive Hamiltonians, \textit{Funkcial. Ekvac.} \textbf{64} (2021), no. 2, 199--223.
\bibitem{JeOz}A. Jensen, T. Ozawa, Classical and quantum scattering for Stark Hamiltonians with slowly decaying potentials, \textit{Ann. Inst. H. Poincar\'e Phys. Th\'eor.} \textbf{54} (1991), no. 3, 229--243.
\bibitem{Ju}W. Jung, Geometrical approach to inverse scattering for the Dirac equation, \textit{J. Math. Phys.} \textbf{38} (1997), no. 1, 39--48.
\bibitem{Ka1}M. Kawamoto, Final state problem for nonlinear Schr\"odinger equations with time-decaying harmonic oscillators, \textit{J. Math. Anal. Appl.} \textbf{503} (2021), no. 1, Paper No. 125292, 17 pp.
\bibitem{Ka2}M. Kawamoto, Asymptotic behavior for nonlinear Schr\"odinger equations with critical time-decaying harmonic potential, \textit{J. Differential Equations} \textbf{303} (2021), 253--267.
\bibitem{Ka3}M. Kawamoto, Strichartz estimates for Schr\"odinger operators with square potential with time-dependent coefficients, \textit{Differ. Equ. Dyn. Syst.} \textbf{31} (2023), no. 4, 877--845.
\bibitem{KaMi1}M. Kawamoto, H. Miyazaki, Long-range scattering for a critical homogeneous type nonlinear Schr\"odinger equation with time-decaying harmonic potentials, \textit{J. Differential Equations} \textbf{365} (2023), 127--167.
\bibitem{KaMi2}M. Kawamoto, H. Miyazaki, Modified scattering operator for nonlinear Schr\"odinger equations with time-decaying harmonic potentials, \textit{Nonlinear Anal.} \textbf{256} (2025), Paper No. 113778.
\bibitem{KaMu}M. Kawamoto, R. Muramatsu, Asymptotic behavior of solutions to nonlinear Schr\"odinger equations with time-dependent harmonic potentials, \textit{J. Evol. Equ.} \textbf{21} (2021), no. 1, 699--723.
\bibitem{KaSa}M. Kawamoto, T. Sato, Asymptotic behavior of solutions to a dissipative nonlinear Schr\"odinger equation with time-dependent harmonic potentials, \textit{J. Differential Equations} \textbf{345} (2023), 418--446.
\bibitem{KaYo}M. Kawamoto, T. Yoneyama, Strichartz estimates for harmonic potential with time-decaying coefficient, \textit{J. Evol. Equ.} \textbf{18} (2018), no. 1, 127--142.
%\bibitem{Ko}E. L. Korotyaev, On scattering in an exterior homogeneous and time-periodic magnetic field, \textit{Math. USSR-Sb.} \textbf{66} (1990), no. 2, 499--522.
%\bibitem{LoHiBe}J. L\H{o}rinczi, F. Hiroshima, V. Betz, Feynman-Kac-type theorems and Gibbs measures on path space, Vol. 1, \textit{De Gruyter, Berlin}, 2020.
\bibitem{Ni1}F. Nicoleau, Inverse scattering for Stark Hamiltonians with short-range potentials, \textit{Asymptotic Anal.} \textbf{35} (2003), 349-359.
\bibitem{Ni2}F. Nicoleau, An inverse scattering problem for short-range systems in a time-periodic electric field, \textit{Math. Res. Lett.} \textbf{12} (2005), 885-896.
\bibitem{Ni3}F. Nicoleau, Inverse scattering for a Schr\"odinger operator with a repulsive potential, \textit{Acta Math. Sin. (Engl. Ser.)} \textbf{22} (2006), no. 5, 1485--1492.
\bibitem{Oz}T. Ozawa, Nonexistence of wave operators for Stark effect Hamiltonians, \textit{Math. Z.} \textbf{207} (1991), no. 3, 335--339
%\bibitem{ReSi1}M. Reed, B. Simon, Methods of Modern Mathematical Physics. II. Fourier analysis, self-adjointness. \textit{Academic Press, New York-London}, 1975.
\bibitem{ReSi2}M. Reed, B. Simon, Methods of Modern Mathematical Physics, III Scattering theory, \textit{Academic Press, New York-London}, 1979.
%\bibitem{St}E. M. Stein, Singular integrals and differentiability properties of functions. Princeton Mathematical Series, No. 30 \textit{Princeton University Press, Princeton, N.J.} 1970.
%\bibitem{Si}B. Simon, Phase space analysis of simple scattering systems: extensions of some work of Enss. \textit{Duke Math. J.} \textbf{46} (1979), no. 1, 110--168.
\bibitem{VaWe}G. D. Valencia, R. Weder, High-velocity estimates and inverse scattering for quantum $N$-body systems with Stark effect, \textit{J. Math. Phys.} \textbf{53} (2012), 102105, 30pp.
\bibitem{Wa1}M. Watanabe, Time-dependent method for non-linear Schr\"odinger equations in inverse scattering problems, \textit{J. Math. Anal. Appl.} \textbf{459} (2018), no. 2, 932--944.
\bibitem{Wa2}M. Watanabe, Time-dependent methods in inverse scattering problems for the Hartree-Fock equation, \textit{J. Math. Phys.} \textbf{60} (2019), no. 9, 091504, 19 pp.
\bibitem{Wa3}M. Watanabe, Inverse $N$-body scattering with the time-dependent Hartree-Fock approximation, \textit{Inverse Probl. Imaging} \textbf{15} (2021), no. 3, 499--517.
\bibitem{We}R. Weder, Multidimensional inverse scattering in an electric field, \textit{J. Funct. Anal.} \textbf{139} (1996), 441-465.
\bibitem{Ya}K. Yajima, Schr\"odinger evolution equations with magnetic fields, \textit{J. Analyse Math.} \textbf{56} (1991), 29--76.
\bibitem{Zo}J. Zorbas, Scattering theory for Stark Hamiltonians involving long-range potentials, \textit{J. Mathematical Phys.} \textbf{19} (1978), no. 3, 577--580.
\end{thebibliography}
\end{document}